\documentclass[11pt, leqno]{amsart}

\usepackage{stmaryrd,graphicx}
\usepackage{amssymb,mathrsfs,amsmath,amscd,amsthm,color, amsfonts}

\usepackage{float}
\usepackage{mathabx,mathdots,textcomp}
\usepackage[all,cmtip]{xy}
\DeclareMathAlphabet{\mathpzc}{OT1}{pzc}{m}{it}
\usepackage{amsfonts,latexsym,wasysym}
\usepackage[normalem]{ulem}
\usepackage[latin1]{inputenc}
\usepackage{cleveref}

\newtheorem{theorem}{Theorem}[section]
\newtheorem{lemma}[theorem]{Lemma}
\newtheorem*{maintheorem*}{Main Theorem}

\newtheorem{corollary}[theorem]{Corollary}
\theoremstyle{definition}\newtheorem{definition}[theorem]{Definition}

\newtheorem{question}[theorem]{Question}
\newtheorem{remark}[theorem]{Remark}

\newcommand{\s}{\operatorname{sup}}

\newcommand {\diam}{\operatorname{diam}}
\newcommand {\conv}{\operatorname{conv}}
\newcommand{\den}{\simeq_{\mathbb R}}

\begin{document}

\title[Tree-like]{Tree-like is not a transitive relation on paths}

\author[J. Brazas]{Jeremy Brazas}
\address{West Chester University\\ Department of Mathematics\\
West Chester, PA 19383, USA}
\email{jbrazas@wcupa.edu}

\author[G.R. Conner]{Gregory R. Conner}
\address{Brigham Young University\\ Department of Mathematics \\
Provo, UT 84602, USA}
\email{conner@math.byu.edu}

\author[P. Fabel]{Paul Fabel}
\address{Mississippi State University\\ Department of Mathematics and Statistics\\
Mississippi State, MS 39762, USA}
\email{fabel@math.msstate.edu}

\author[C. Kent]{Curtis Kent}
\address{Brigham Young University\\ Department of Mathematics \\
Provo, UT 84602, USA}
\email{curtkent@mathematics.byu.edu}

\subjclass[2010]{ 54F50, 55R65, 57M10 }
\keywords{$\mathbb R$-tree, geodesic $\mathbb R$-tree reduction, path-homotopy, dendrite, unique path lifting, covering map}
\date{\today}

\begin{abstract}
The notions of \emph{tree-like loop} and \emph{Lipschitz tree-like loop} were introduced by Hambly and Lyons in their 2010 Annals of Mathematics paper. They showed that the Lipschitz tree-like property determines an equivalence relation on the set of paths of bounded variation in a given metric space and then asked if this notion could be extended to paths without the Lipschitz requirement. We show that after eliminating the Lipschitz requirement, the resulting relation is no longer transitive and thus is not an equivalence relation. The counterexample is obtained by analyzing an explicit fractal construction in the plane.
\end{abstract}

\begingroup
\def\uppercasenonmath#1{} 
\let\MakeUppercase\relax 
\maketitle
\endgroup

\section{Introduction}

In their 2010 Annals of Mathematics paper, Hambly and Lyons considered a notion of a path  with bounded variation being \emph{tree-like} \cite{HamblyLyons2010}.  Their definition encoded $\mathbb R$-trees using positive continuous functions to the real line, called height functions.

\begin{definition}[\cite{HamblyLyons2010}]
    A continuous function $\alpha:[0,1]\to X$ in a metric space $(X,d)$ is \emph{tree-like} if there exists a positive real-valued continuous function $h:[0,1]\to\mathbb R$  such that, for all $s\leq t$, \[d\big(\alpha(t), \alpha(s)\big) \leq h(t)+h(s)-2\inf\limits_{u\in[s,t]} h(u).\]

    The function $h$ is a \emph{height function} for $\alpha$. Two paths $\alpha,\beta:[0,1]\to X$ are \emph{tree-like equivalent}, if $\alpha*\overline\beta$ is tree-like.  We say $\alpha$ is a \emph{Lipschitz tree-like} path if $\alpha$ has a height function with bounded variation.  Two paths $\alpha,\beta:[0,1]\to X$ are \emph{Lipschitz tree-like equivalent}, if $\alpha*\overline\beta$ is Lipschitz tree-like.
\end{definition}

Hambly and Lyons showed that \emph{Lipschitz tree-like} is an equivalence relation on the set of paths with bounded variation and that each equivalence class has a unique (up to reparameterization) representative of shortest length.  They then showed that, for paths of bounded variation in $\mathbb R^d$, the signature of a path is trivial if and only if the path is Lipschitz tree-like \cite[Theorem 4]{HamblyLyons2010}.  Hambly and Lyons asked how important was the finite length condition and posed the following two questions.


\begin{question}\cite[Problem 1.8]{HamblyLyons2010}
    Does tree-like equivalent (without requiring the height function to be Lipschitz) define an equivalence relation on the set of continuous paths?
\end{question}

\begin{question}\cite[Problem 1.9]{HamblyLyons2010}
    Is there a unique tree reduced path associated to any continuous path?
\end{question}

The main result of this paper, Theorem \ref{finalthm}, gives a negative answer to both of these questions. To prove this theorem, we require an alternative definition of ``tree-like" that was introduced by  Boedihardjo, Geng, Lyons, and Yang in \cite{BoedihardjoGengLyonsYang2016} and which applies to paths that do not have bounded variation. The same definition was considered independently in \cite{BrazasConnerFabelKent_preprint} where it was called ``$\mathbb R$-tree nullhomotopic." To clarify the distinction for general paths, we will use this latter term (see Definition \ref{defn: R-tree nullhomotopic}). Boedihardjo et al.~ were able to show that weakly geometric rough paths have trivial signature if and only if the path is $\mathbb R$-tree nullhomotopic.


\begin{definition}[\cite{BoedihardjoGengLyonsYang2016, BrazasConnerFabelKent_preprint}]\label{defn: R-tree nullhomotopic}
    We say that $\alpha: [0,1]\to X$ is \emph{$\mathbb R$-tree nullhomotopic}, if $\alpha$ factors through a loop in an $\mathbb R$-tree, that is, if there exists an $\mathbb{R}$-tree $E$, a loop $\pi:[0,1]\to E$ and a map $g:E\to X$ such that $g\circ\pi=\alpha$.

    We say that paths $\alpha,\beta:[0,1]\to X$ are \emph{$\mathbb R$-tree homotopic}, and we write $\alpha\den\beta$, if $\alpha*\overline{\beta}$ is $\mathbb R$-tree nullhomotopic.
\end{definition}

A principle result of \cite{BrazasConnerFabelKent_preprint} (Theorem 1.4) is that $\den$ is not an equivalence relation on the space of paths, specifically $\den$ does not satisfy transitivity.  However, this in and of itself does not answer the questions of Hambly and Lyons, as it was not demonstrated that the paths causing non-transitivity admit a height function.  While it  is possible to work through the construction there to show that constructed paths admit height functions, here we will present an alternative geometric construction where it is immediate that the constructed loops admit height functions.  This will also provide an independent proof of Theorem 1.4 of \cite{BrazasConnerFabelKent_preprint}.

The rest of the paper is structured as follows. In Section \ref{sectionrelationbwtypes}, we show that every tree-like loop is $\mathbb{R}$-tree nullhomotopic and that the converse holds when the factorization through the $\mathbb{R}$-tree $E$ can be achieved with a Lipschitz map $g:E\to X$. In Section \ref{sectionbuildingdendrites}, we show how to modify certain piecewise-linear paths that factor through a simplicial tree.   In Section \ref{sec: construction}, we inductively apply the construction from Section \ref{sectionbuildingdendrites} to build two sequences of paths, $g_n\circ \pi_n$ and $\tilde g_n\circ\tilde \pi_n$ that factor through successively larger simplicial trees. In Section \ref{sectionlimiting}, we show that these two sequences of paths limit on the same path $\gamma$, which is space-filling in the convex hull of a right triangle\footnote{After constructing $\gamma$ independently, the authors recognized that $\gamma$ is a parameterization of the Sierpsi\'nski Curve and that our approximation of it is essentially Knopp's representation \cite[Section 4.2]{Sagan1994}}. We also identify general metric conditions ensuring that constructions such as that given in Section \ref{sec: construction} result in $\mathbb{R}$-tree homotopic paths. In Section \ref{sec: the bow on top}, we give a negative answer to Hambly and Lyons' first question by showing that $\alpha=g_1\circ\pi_1$ is tree-like equivalent to $\gamma$, $\beta=\tilde g_1\circ\tilde \pi_1$ is tree-like equivalent to $\gamma$, but that $\alpha$ is not tree-like equivalent to $\beta$. Since the paths $\alpha$ and $\beta$ are rectifiable arcs which only intersect at their ends, they will be reduced and hence give a negative answer to Hambly and Lyons' second question.

\section{Relation between types of tree-like}\label{sectionrelationbwtypes}
We will first give a short proof that tree-like implies $\mathbb R$-tree nullhomotopic. The interested reader can also see \cite{Duquesne2006,HamblyLyonsNotes,  LeGall1991} for additional proofs. Recall that a topological space underlying a compact $\mathbb{R}$-tree is called a \textit{dendrite}. Equivalently, a dendrite is a uniquely arcwise connected Peano continuum.

\begin{definition}[Dendrites from height functions]
    For any continuous function $h:[0,1]\to\mathbb R$, we can define an equivalence relation $\sim_h$ on $[0,1]$ by $s\sim_h t$, if $h(s)= h(t)=a$ and $s,t$ are contained in a single component of $h^{-1}\big([a,\infty)\big)$.  Note, for  $s<t$, this is equivalent to $a= h(s) = h(t) = \inf\limits_{u\in[s,t]} h(u)$.  Then, for any $t\in [0,1]$, the  equivalence class of $t$ is $[t]_h = C\cap h^{-1}\big(h(t)\big)$ where $C$ is the component of  $h^{-1}\big([h(t), \infty) \big)$ containing $t$.      Let $\pi: [0,1]\to E$ be the decomposition map where $E = [0,1]/\hspace{-.25em}\sim_h$. 
\end{definition}
     The following two properties are an immediate consequence of the definition.

    \begin{enumerate}
        
    \item[(*)]A component $C$ of $h^{-1}\big([a,\infty)\big)$ is closed and $\pi$-saturated.  Thus $[0,1]\backslash C$ is open and $\pi$-saturated.
    \end{enumerate}  
    
     \begin{enumerate}
        \item[(**)]For every component $C$ of $h^{-1}\big([a,\infty)\big)$, the components of $C\backslash h^{-1}(a)$ are open and $\pi$-saturated.
    \end{enumerate}

\begin{lemma}\label{lem: build denrite}
    The equivalence relation $\sim_h$ induces an upper semi-continuous decomposition of $[0,1]$ such that the quotient $\pi: [0,1]\to E=[0,1]/\hspace{-.25em}\sim_h$ is a dendrite.  
\end{lemma}

    \begin{proof}
        By Theorem 10.2 in \cite{Nadler1992}, we need only show that $E$ is a continuum and that for any two points in $E$ there is a third point separating them in $E$.  If $E$ is Hausdorff, then $\pi$ is an upper semi-continuous decomposition and $E$ is a continuum.  For any $t\in[0,1]$,  $[t]_h = C\cap h^{-1}\big(h(t)\big)$ is the intersection of closed sets where $C$ is the component of  $h^{-1}\big([h(t), \infty) \big)$ containing $t$.  Thus equivalence classes are closed.  Therefore, we need only show that for every two points in $E$ there is a third point separating them.

    Suppose that $e= [r]_h$ and $e'= [s]_h$ are distinct points of $E$.    We will prove that their exists a $t\in [0,1]$ such that $e$ and $e'$ lie in distinct open components of $E\backslash \big\{\pi(t)\big\}$. We will now consider two cases.  

    \textbf{Case 1:}  \emph{The component $C_r$ of $h^{-1}\big([h(r), \infty) \big)$ containing $r$ and the component $C_s$ of $h^{-1}\big([h(s), \infty) \big)$ containing $s$ are disjoint.}  Let $[a,b]$ be the interval with endpoints in $C_r\cup C_s$ and with $(a,b)\cap (C_r\cup C_s)=\emptyset$. Fix $t\in [a,b]$ such that $h(t) = \inf h\bigl([a,b]\big)$.  Notice $h(t) < \min \big\{h(r), h(s)\big\}$.  Then the component $C_t$ of  $h^{-1}\big([h(t), \infty) \big)$ containing $t$ contains $C_r\cup C_s$ and $[r]_h$, $[s]_h$ are contained in distinct components of $C_t\backslash [t]_h$.  Then $[0,1]\backslash [t]_h = [0,1]\backslash C_t \cup  \Big(\bigcup\limits_{C\in \Lambda_t} C\Big)$ where $\Lambda_t $ is the set of components of $C_t\backslash [t]_h$.

    
    Thus  $\Big\{\pi\big([0,1]\backslash [t]_h\big), \pi (C)\mid C\in \Lambda_t\Big\}$ is a collection of disjoint open subsets of $E\backslash \pi(t)$  and   $\pi(r)$ and $\pi(s)$ are contained in distinct elements of this set.  Thus $e$ and $e'$ are contained in distinct open components of $E\backslash \pi(t)$.
    
    \textbf{Case 2:}  \emph{The component $C_r$ of $h^{-1}\big([h(r), \infty) \big)$ containing $r$ and the component $C_s$ of $h^{-1}\big([h(s), \infty) \big)$ containing $s$ are not disjoint.}  We may, without loss of generality, assume that $h(r) \leq h(s)$.  Then, since $C_r\cap C_s \neq \emptyset$, we have that $C_s \subset C_r$ and $h(r)< h(s)$.  Then $[r]_h\cap C_s = \emptyset$.  Fix $[a,b]\subset C_r$ such that $[a,b]\cap [r]_h= \{a\}$ and $[a,b]\cap C_s = \{b\}$.  Let $t\in [\frac{a+b}{2},b]$ such that $h(t) =\inf h\big([\frac{a+b}{2},b]\bigr)$.  Then $h(r)< h(t)< h(s)$.  Thus, for $C_t$ the component of $h^{-1}\big([h(t), \infty) \big)$ containing $t$, we have that $C_s\subset C_t$ and $[r]_h \subset [0,1]\backslash C_t$.  Thus $\pi\big([0,1]\backslash C_t\big)$ and $\pi\big(C_t\backslash [t]_h\big) $ give an open separation of $E\backslash \{\pi(t)\}$ and contain the elements $\pi(r)$ and $\pi (s)$ respectively.

     \end{proof}

\begin{lemma}\label{lem: dendrites from height functions}
    If $\alpha:[0,1]\to X$ is tree-like, then $\alpha$ is $\mathbb R$-tree nullhomotopic.
\end{lemma}    

    \begin{proof}
        Since $\alpha$ is tree-like, there exists a positive real-valued continuous function $h:[0,1]\to\mathbb R$  such that, for all $s\leq t$, $d\big(\alpha(t), \alpha(s)\big) \leq h(t)+h(s)-2\inf\limits_{u\in[s,t]} h(u)$.  By Lemma \ref{lem: build denrite}, we have a continuous function $\pi :[0,1]\to E$ to a dendrite induced by $\sim_h$.  Notice that if $s\sim_h t$ and $s\leq t$, then $d\big(\alpha(t), \alpha(s)\big) \leq h(t)+h(s)-2\inf\limits_{u\in[s,t]} h(u)= 0$.  Thus, by the universal property of quotient maps, there exists a continuous map $g:E \to X$ such that $g\circ \pi= \alpha$. 
    \end{proof}

\begin{lemma} [Height functions from dendrites]\label{lem: height functions from dendrites}
Let $(X,d)$ be a metric space and $(E,d')$ be a dendrite with a geodesic metric $d'$. Suppose that there exists a loop $\pi:[0,1]\to E$ and a Lipschitz map $g: E\to X$, then $g\circ\pi$ is tree-like.  In particular, $h:[0,1]\to \mathbb R$ given by $h(t) = Ld\big(\pi(0), \pi(t)\big)$, where $L$ is the Lipschitz constant for $g$, is a height function for $g\circ\pi$.
\end{lemma}

\begin{proof}
    Let $\pi$, $g$, and $h$ be as stated in the lemma.  Fix $s,t\in [0,1]$ and suppose that $s\leq t$.  Let $A_s$, $A_t$ be the arc in $E$ from $\pi(0)$ to $\pi(s)$, $\pi(t)$, respectively.  Let $x$ be the unique point on $A_s\cap A_t$ that separates $\pi(s)$ and $\pi(t)$.  Since $x$ separates $\pi(s)$ and $\pi(t)$,   we have that $\inf\limits_{u\in[s,t]} h(u)\leq Ld\big(\pi(0),x\big)$.  Since $x\in A_s\cap A_t$, $d'\big(\pi(0), \pi(u)\big) = d'\big(\pi(0), x\big)+d'\big(x, \pi(u)\big)$ for $u\in \{s,t\}$.   
    Then
    
    \begin{align*}
        h(t) + h(s) - 2\inf\limits_{u\in[s,t]} h(u)&\geq h(t) + h(s) - 2Ld'\big(\pi(0), x\big) \\
        &= Ld'\big(\pi(0), \pi(t)\big)+Ld'\big(\pi(0), \pi(s)\big)- 2Ld'\big(\pi(0), x\big)\\
        &=Ld'\big(x, \pi(t)\big)+Ld'\big(x, \pi(s)\big)\\
        &\geq d\big(g(x), g\circ\pi(t)\big)+d\big(g(x), g\circ\pi(s)\big)\geq d\big(g\circ\pi(t), g\circ\pi(s)\big)
    \end{align*}
    Thus $h$ is a height function for $g\circ\pi$.
    
\end{proof}

\begin{remark}
    Since we did not require $\pi$ to be Lipschitz, Lemma \ref{lem: height functions from dendrites} does not imply $g\circ \pi$ is Lipschitz tree-like.
\end{remark}


\section{Building dendrites and maps}\label{sectionbuildingdendrites}

A \emph{metric simplicial tree} is a simplicial tree endowed with an edge metric given by an assignment of length to each edge.  Let $E$ be a metric simplicial tree with vertex set $E^{(0)}$. A continuous map $\pi:[0,1]\to E$ is \emph{simplicial} if all point preimages are finite and $\pi$ is linear on the closure of each component of $[0,1]\backslash\pi^{-1}(E^{(0)})$. A continuous map $g: E\to \mathbb R^2$ is \emph{simplicial} if all point preimages are finite and $g$ is linear on each edge of $E$.

\subsection{Parameterizing edges of a triangle}

Let $T$ be an isosceles right triangle in the Euclidean plane with vertices $\big\{a, b, c\big\}$, where the right angle is at $b$, and with given ordering of vertices $a,b,c$.  We write $\triangle abc$ for $T$ when we wish to refer to this vertex ordering. Suppose that we have simplicial trees $E$ and $\tilde E$, an interval $[r,s]$,  simplicial maps $\pi: [0,1]\to E$, $\tilde \pi: [0,1]\to \tilde E$,  $g: E\to \mathbb R^2$, and $\tilde g: \tilde E\to \mathbb R^2$  such that
\begin{enumerate}
    \item\label{ind: vertices} $\pi|_{[r,s]}$, $\tilde \pi|_{[r,s]}$ are injective with $\pi\big(\{r,\tfrac{r+2}{2}, s\}\bigr)\subset E^{(0)}$, $\tilde \pi\big(\{r,s\}\bigr)\subset \tilde E^{(0)}$

    \item $g\circ \pi(r) = \tilde g\circ \tilde \pi(r) =a$, $g\circ \pi(\frac{r+s}{2}) = b$, and  $g\circ \pi(s) = \tilde g\circ \tilde \pi(s) = c$; and
    \item $g\circ \pi|_{[r,\frac{r+s}{2}]}$, $g\circ \pi|_{[\frac{r+s}{2}, s]}$, and $\tilde g\circ \tilde \pi|_{[r,s]}$ are linear maps.

\end{enumerate}

\noindent In this situation, we will say $(\pi,\tilde \pi, g,\tilde g)$ \emph{parametrizes $\triangle abc$ on $[r,s]$}; see Figure \ref{fig:basecse}. Note that the paths $\pi,\tilde\pi$ match the ordering of $\triangle abc$. Additionally, note that it is possible that arcs $\pi\big([r,s]\big)$, $\tilde \pi\big([r,s]\big)$ contain vertices in addition to those specified by Condition (\ref{ind: vertices}).

\begin{figure}[ht]
    \centering
        {\tiny\def\svgwidth{.75\textwidth}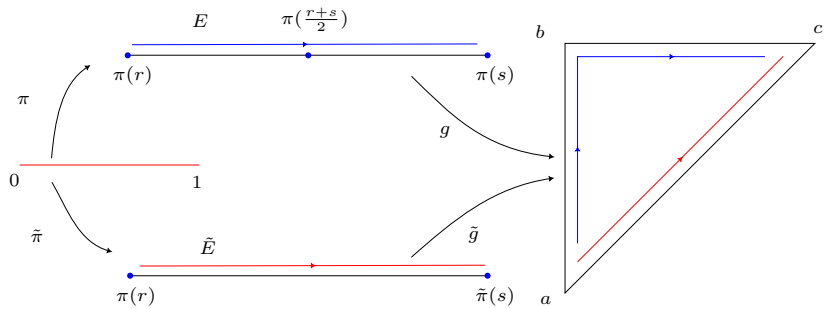}
    \caption{Parameterizing $\triangle abc$ on $[r,s]$}
    \label{fig:basecse}
\end{figure}

\subsection{The subdivision step}\label{sec: subdivision}
If $(\pi,\tilde \pi, g,\tilde g)$ parametrizes $T=\triangle abc$ on $[r,s]$, then we can construct a new simplicial tree  $E_1$  by making  $\pi(\frac{3r+s}{4})$ and $\pi(\frac{r+3s}{4})$ vertices, if they weren't already.  We will endow the subdivided edges with lengths such that the metric remains unchanged on $E$.  We will then add two new edges at these points with length half the length of the arc from $\pi(r)$ to $\pi(\frac{r+s}{2})$; see Figure \ref{fig:step1}. We construct $\tilde{E_1}$ similarly by making $\tilde \pi(\frac{r+s}{2})$ a vertex (if it wasn't already) and attaching a new edge at this (possibly new) vertex.  We will endow the subdivided edges with lengths such that the metric on $\tilde E$ remains unchanged and the new edge will have length half the length of the arc from $\tilde \pi(r)$ to $\tilde \pi(s)$.

Then we can naturally identify $E$, $\tilde E$ as a metric subspace of $E_1$, $\tilde E_1$, respectively, and define  retractions  $\rho_1: E_1\to E$, $\tilde\rho_1:\tilde E_1\to \tilde E$  which collapse the added edges to their vertex in $E$ or $\tilde E$, respectively.  Notice that $\rho, \tilde \rho$ are monotone retracts and the sup-metric distances $d^{\s}(\rho_1, id_{E_1})$ and $d^{\s}(\tilde \rho_1, id_{\tilde E_1})$ are equal to the diameter of the added edges.

\begin{figure}[ht]
    \centering
        {\tiny\def\svgwidth{\textwidth}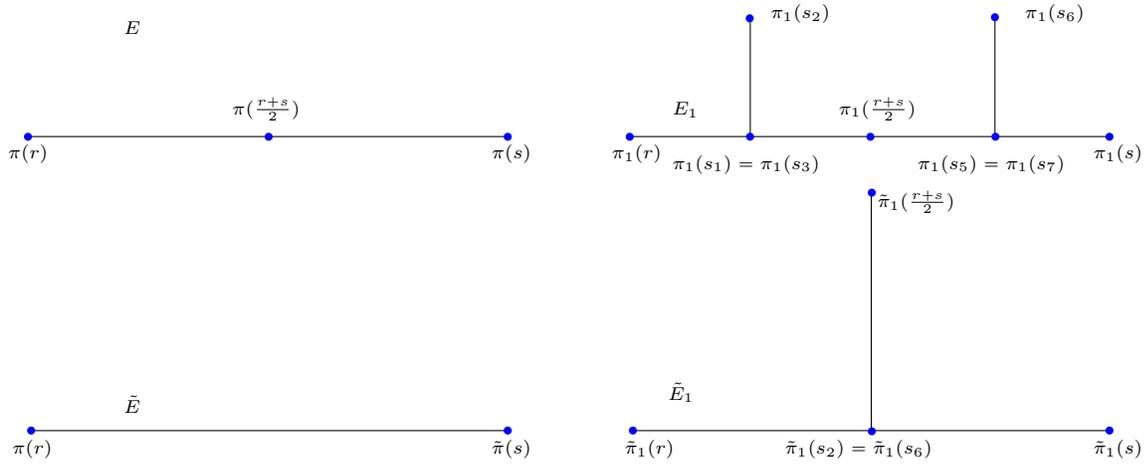}
    \caption{Adding and subdividing edges}
    \label{fig:step1}
\end{figure}

We can then define maps $\pi_1: [0,1]\to E_1$ and $\tilde \pi_1: [0,1]\to \tilde E_1$ by $\pi_1|_{[0,1]\backslash (r,s)} = \pi|_{[0,1]\backslash (r,s)}$, $\tilde \pi_1|_{[0,1]\backslash (r,s)} = \tilde \pi|_{[0,1]\backslash (r,s)}$,  and $\pi_1$, $\tilde \pi_1$ are linear maps as illustrated in Figure \ref{fig:inductive step} on $[r,s]$ where $s_i = \frac{(8-i)r + i\cdot s}{8}$.

Let $d=\frac{a+b}{2}$, $e=\frac{b+c}{2}$, and $f=\frac{a+c}{2}$ denote the midpoints of the sides of $T$. We define $g_1: E_1 \to \mathbb{R}^2$ and $\tilde g_1: \tilde E_1 \to \mathbb{R}^2$ by $g_1|_{E} = g$, $\tilde g_1|_{\tilde E} = \tilde g$, $g_1\circ \pi_1(s_2)=g_1\circ \pi_1(s_6) =f $, and $\tilde g_1\circ \tilde \pi_1(\frac{r+s}{2})= b$ and then extend linearly. Notice that $(\pi_1,\tilde \pi_1, g_1,\tilde g_1)$ parametrizes $\triangle adf$ on $[r,\frac{3r+s}{4}]$, $\triangle fdb $ on $[\frac{3r+s}{4},\frac{r+s}{2}]$, $\triangle bef$ on $[\frac{r+s}{2},\frac{r+3s}{4}]$, and $\triangle fec$ on $[\frac{r+3s}{4},s]$ where $\triangle adf$, $\triangle fdb$, $\triangle bef$, and $\triangle fec$ are vertex-ordered isosceles right triangles each with diameter a half that of $T$. The order that $\pi$ traverses the legs of these right triangles will induce an ordering on $\{\triangle adf,\triangle fdb,  \triangle bef, \triangle fec\}$, which we will refer to as the \textit{natural ordering inherited from} $\pi$.

\begin{figure}[ht]
    \centering
        {\tiny\def\svgwidth{\textwidth}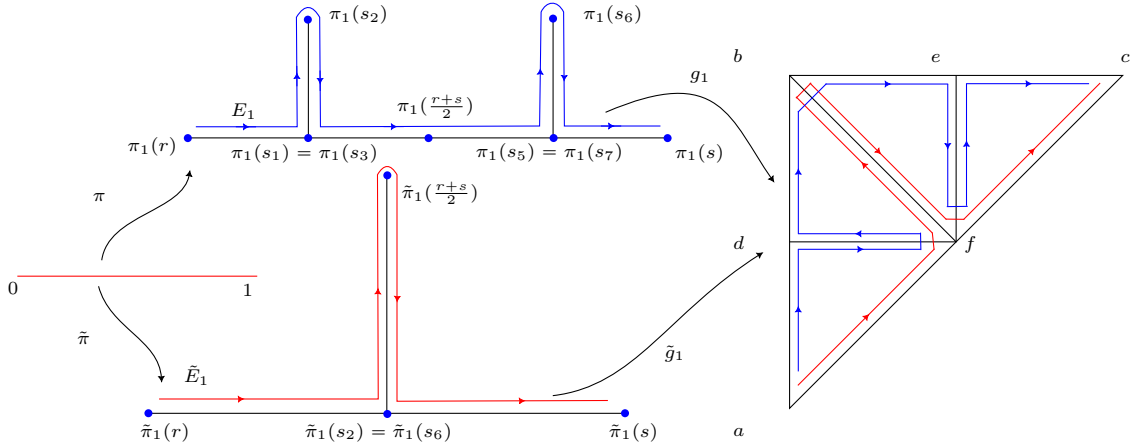}
    \caption{The maps}
    \label{fig:inductive step}
\end{figure}

\section{The construction}\label{sec: construction}
Let $A =\{T_1, \dots, T_n\}$ be an ordered collection of isosceles right triangles with given vertex orderings.  We will say that $(\pi, \tilde \pi, g, \tilde g) $ is a \emph{parametrization of $A$ (relative to $S=\{0=s_0< s_1< \cdots< s_n=1\}$)} if  $(\pi, \tilde \pi, g, \tilde g) $ parametrizes $T_i$ on $[s_{i-1},s_i]$.

In the remainder of the paper, we let $T$ denote the vertex-ordered triangle $\triangle abc$ with vertices $a=(0,0)$, $b=(0,1)$, and $c=(1,1)$. Set $A_1=\{T\}$. Let $E_1$ be the simplicial tree with three vertices $\{v_0, v_\frac12, v_1\}$ and two edges $\big\{(v_0, v_\frac12), (v_\frac12, v_1)\bigr\}$ and $\tilde E_1$ be the simplicial tree with two vertices $\{\tilde v_0, \tilde v_1\}$ and one edge $\big\{(\tilde v_0, \tilde v_1)\bigr\}$. We will endow $E_1$ with the edge metric where both edges have length $1$ and $\tilde E_1$ with the edge metric where its unique edge has length $\sqrt{2}$.

Define $\pi_1: [0,1]\to E_1$, $\tilde\pi_1: [0,1]\to \tilde E_1$ by $\pi_1(i) = v_i$ for $i\in \{0, \frac12, 1\}$, $\tilde\pi_1(i) = \tilde v_i$ for $i\in \{0, 1\}$. Define $g_1: E_1\to \mathbb{R}^2$, $\tilde g_1: \tilde E_1\to \mathbb{R}^2$ by $g_1(v_0) = \tilde g_1(\tilde v_0) = a$, $g_1(v_\frac12)  = b$, and $g_1(v_1) = \tilde g_1(\tilde v_1) = c$ and extending linearly.  Then $(\pi_1, \tilde \pi_1, g_1, \tilde g_1)$ parametrizes $A_1$ on $[0,1]$.

Suppose that we have inductively constructed maps $\pi_n: [0,1]\to E_n$, $\tilde\pi_n: [0,1]\to \tilde E_n$ and  $g_n: E_n\to \mathbb{R}^2$, $\tilde g_n: \tilde E_n\to \mathbb{R}^2$ such that the following hold:

\begin{enumerate}
    \item $E_n$, $\tilde E_n$ are simplicial trees with edge metrics where all edges have length $\frac{1}{2^{n-1}}$, $\frac{\sqrt{2}}{2^{n-1}}$, respectively.
    \item $E_n$, $\tilde E_n$ are monotone retracts of $E_{n-1}$, $\tilde E_{n-1}$, respectively.
    \item\label{collapse} For $n>1$, $\rho_n: E_{n}\to E_{n-1}$, $\tilde \rho_n:\tilde  E_{n}\to\tilde  E_{n-1}$ are monotone retractions such that $d^{\s}(\rho_n, id_{E_n})=\frac{1}{2^{n-1}}$ and $d^{\s}(\tilde \rho_n, id_{\tilde E_n})=\frac{\sqrt2}{2^{n-1}}$.
    \item $g_n: E_n\to \mathbb{R}^2$ and $\tilde g_n: \tilde E_n\to \mathbb{R}^2$ are $1$-Lipschitz.
    \item For $n>1$, $g_n|_{E_{n-1}}= g_{n-1}$ and $\tilde g_n|_{\tilde E_{n-1}}= \tilde g_{n-1}$.
    \item $A_n$ is an ordered collection of $4^{n-1}$ vertex-ordered isosceles right triangles with leg lengths $\frac{1}{2^{n-1}}$.
    \item $(\pi_n, \tilde \pi_n, g_n, \tilde g_n) $ is a parametrization of $A_n$ relative to $S=\{0, \frac{1}{4^{n-1}}, \frac{2}{4^{n-1}}, \dots, 1\}$
\end{enumerate}

We are now ready to define $A_{n+1}$, $E_{n+1}$, $\tilde E_{n+1}$, $\pi_{n+1}: [0,1]\to E_{n+1}$, $\tilde\pi_{n+1}: [0,1]\to \tilde E_{n+1}$ and $g_{n+1}: E_{n+1}\to \mathbb{R}^2$, $\tilde g_{n+1}: \tilde E_{n+1}\to \mathbb{R}^2$ that satisfy the inductive hypotheses.

Write $A_n=\{T_1, \dots, T_{4^{n-1}}\}$. Since $(\pi_n, \tilde \pi_n, g_n, \tilde g_n) $ parameterizes $T_i$ on $[\frac{i-1}{4^{n-1}},\frac{i}{4^{n-1}}]$, we can apply the subdivision process of Section \ref{sec: subdivision} for each $i$ to define $A_{n+1}$, $E_{n+1}$, $\tilde E_{n+1}$, $\pi_{n+1}: [0,1]\to E_{n+1}$, $\tilde\pi_{n+1}: [0,1]\to \tilde E_{n+1}$ and  $g_{n+1}: E_{n+1}\to \mathbb{R}^2$, $\tilde g_{n+1}: \tilde E_{n+1}\to \mathbb{R}^2$.  In this process, every edge of $E_n$ (and $\tilde E_n$) is subdivided in half and at each new vertex one or two new edges of length $\frac{1}{2^n}$ are added.  It is then an exercise to see this satisfy the inductive hypothesis. Note that the ordered set $A_{n+1}=\{T_1', \cdots, T_{4^n}'\}$  is obtained from $A_n$ by replacing each $T_i$ with the sequence $T_{4(i-1) +1}',T_{4(i-1) +2}',T_{4(i-1) +3}',T_{4(i-1) +4}'$ consisting of the four vertex-ordered triangles from the subdivision process and ordered with the natural ordering inherited from $\pi_n$.

\begin{figure}[h]
    \centering
        {\tiny\def\svgwidth{\textwidth}
\begingroup%
  \makeatletter%
  \providecommand\color[2][]{%
    \errmessage{(Inkscape) Color is used for the text in Inkscape, but the package 'color.sty' is not loaded}%
    \renewcommand\color[2][]{}%
  }%
  \providecommand\transparent[1]{%
    \errmessage{(Inkscape) Transparency is used (non-zero) for the text in Inkscape, but the package 'transparent.sty' is not loaded}%
    \renewcommand\transparent[1]{}%
  }%
  \providecommand\rotatebox[2]{#2}%
  \newcommand*\fsize{\dimexpr\f@size pt\relax}%
  \newcommand*\lineheight[1]{\fontsize{\fsize}{#1\fsize}\selectfont}%
  \ifx\svgwidth\undefined%
    \setlength{\unitlength}{769.86625671bp}%
    \ifx\svgscale\undefined%
      \relax%
    \else%
      \setlength{\unitlength}{\unitlength * \real{\svgscale}}%
    \fi%
  \else%
    \setlength{\unitlength}{\svgwidth}%
  \fi%
  \global\let\svgwidth\undefined%
  \global\let\svgscale\undefined%
  \makeatother%
  \begin{picture}(1,0.38722693)%
    \lineheight{1}%
    \setlength\tabcolsep{0pt}%
    \put(0,0){\includegraphics[width=\unitlength,page=1]{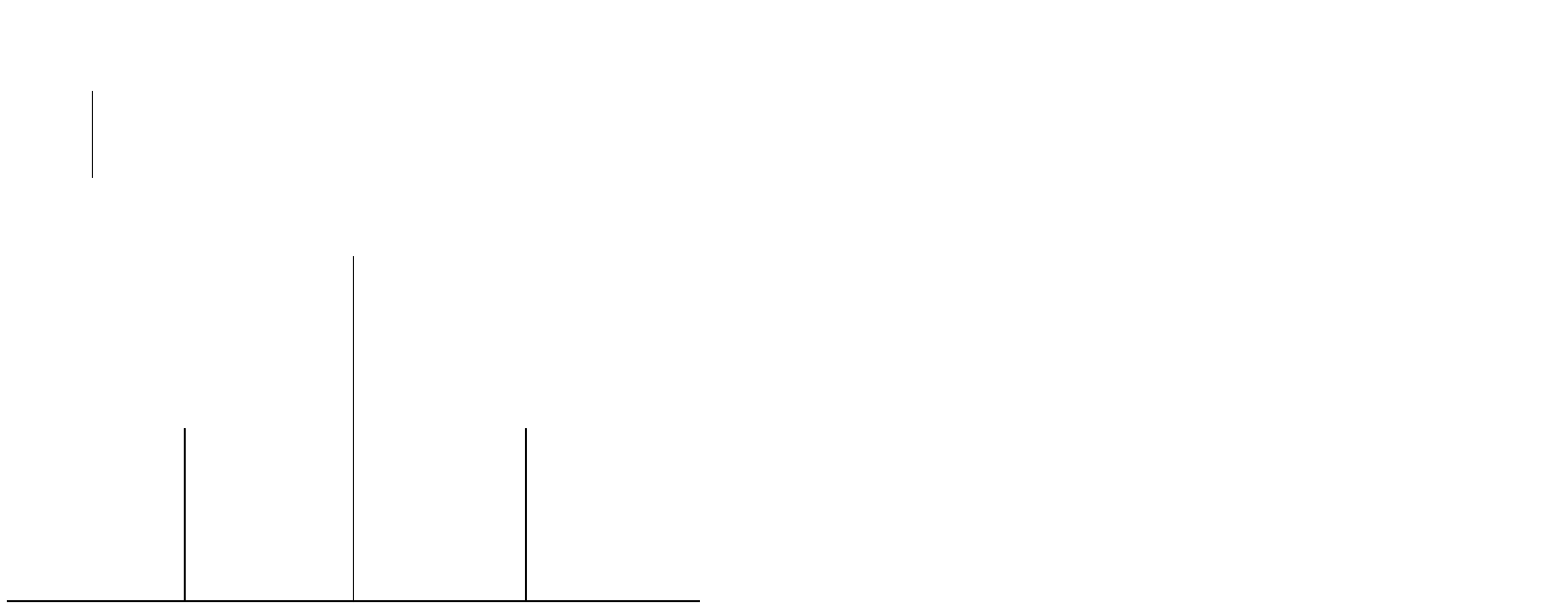}}%
    \put(0.0121819,0.09095919){\color[rgb]{0,0,0}\makebox(0,0)[lt]{\lineheight{1.25}\smash{\begin{tabular}[t]{l}$\tilde E_2$\end{tabular}}}}%
    \put(0,0){\includegraphics[width=\unitlength,page=2]{TreeSteps.pdf}}%
    \put(0.0085253,0.34822082){\color[rgb]{0,0,0}\makebox(0,0)[lt]{\lineheight{1.25}\smash{\begin{tabular}[t]{l}$E_2$\end{tabular}}}}%
    \put(0,0){\includegraphics[width=\unitlength,page=3]{TreeSteps.pdf}}%
    \put(0.53464823,0.34822082){\color[rgb]{0,0,0}\makebox(0,0)[lt]{\lineheight{1.25}\smash{\begin{tabular}[t]{l}$E_3$\end{tabular}}}}%
    \put(0,0){\includegraphics[width=\unitlength,page=4]{TreeSteps.pdf}}%
    \put(0.52067073,0.0914063){\color[rgb]{0,0,0}\makebox(0,0)[lt]{\lineheight{1.25}\smash{\begin{tabular}[t]{l}$\tilde E_3$\end{tabular}}}}%
    \put(0,0){\includegraphics[width=\unitlength,page=5]{TreeSteps.pdf}}%
  \end{picture}%
\endgroup%
}
    \caption{On the left are $E_2$ and $\tilde E_2$.  On the right are $E_3$ and $\tilde E_3$.}
    \label{fig:Tree steps}
\end{figure}

\newpage

\begin{figure}[h]
    \centering
        {\tiny\def\svgwidth{\textwidth}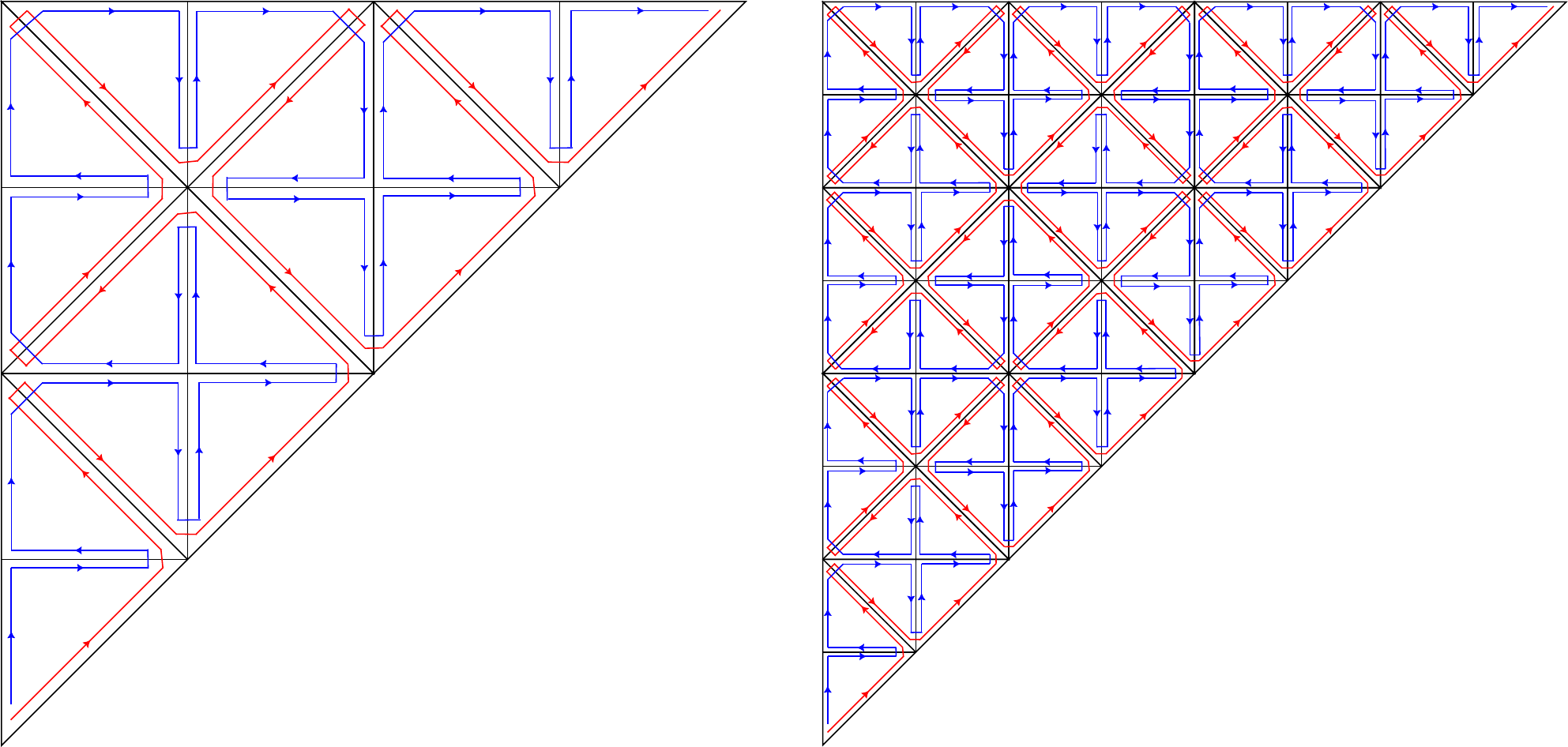}
    \caption{On the left is a schematic for $g_2\circ \pi_2$ (in blue) and $\tilde g_2\circ \tilde \pi_2$ (in red).
    On the right is a schematic for $g_3\circ \pi_3$ (in blue) and $\tilde g_3\circ \tilde \pi_3$ (in red).}
    \label{fig:Step3}
\end{figure}

\section{Limiting arguments}\label{sectionlimiting}

\begin{lemma}\label{lem: the dendrite}
Suppose that $E_i$ is a nested sequence of dendrites with geodesic metrics $d_i$ such that $d_{i}|_{E_{i-1}\times E_{i-1}} = d_{i-1}$ and suppose $a>0$. If, for all $i>1$ , $\rho_i:E_i\to E_{i-1}$ is a monotone retract and $\sup\limits_{e\in E_i}d_i\big(\rho_i(e), e\big)\leq \frac{a}{2^i}$, then $E = \varprojlim(E_i, \rho_i)$ is a dendrite and  $E$ admits a metric $d$ such that  $d|_{E_i\times E_i}= d_i$ (where we identify $E_i$ with its natural embedding into $E$) and the projections $\varrho_i: E\to E_i$ converge uniformly to the identity map on $E$.
\end{lemma}

\begin{proof}
    It follows from Theorem 10.36 in \cite{Nadler1992} that $E$ is a dendrite.  We need only show that $E$ admits a metric satisfying the two conditions of the lemma.

    For $(e_i), (f_i)\in E $, let  $d\big((e_i), (f_i)\big) = \sup_i d_i(e_i, f_i)$.  Since $\sup\limits_{e\in E_i}d_i\big(\rho_i(e), e\big)\leq \frac{a}{2^i}$,

    \begin{align*}
        d_i(e_i, f_i) \leq& \sum\limits^{i}_{j=2}d_j(e_j, e_{j-1}) + d(e_1, f_1)+ \sum\limits^{i}_{j=2} \big( d_j(f_j, f_{j-1})  \\
        \leq & \sum\limits_{j=1}^\infty \frac{a}{2^j} + \diam(E_1) + \sum\limits_{j=1}^\infty \frac{a}{2^j} = \diam(E_1)+2a,
    \end{align*} which implies that $d$ (the uniform metric on $E$ as a subset of $\prod_i E_i$) is well-defined.

     The uniform topology on a product of metric spaces is always finer than the product topology so it suffices to show that the inverse limit topology is finer than the topology induced by $d$. Let $(e_n)\in E$ and $\epsilon>0$ and consider the open $d$-ball $B_{d}((e_n),\epsilon)$ in the metric space $(E,d)$. Find $m\in\mathbb N$ such that $\sum_{j=m+1}^{\infty}\frac{a}{2^j}<\frac{\epsilon}{4}$. Let $U_n=B_{d_n}(e_n,\epsilon/4)$ for $n\in\{1,2,\dots, m\}$ and let $U_n=E_n$ for $n>m$. We claim that $E\cap \prod_{n\in\mathbb N}U_n\subseteq B_{d}((e_n),\epsilon)$. Given $(f_n)\in E\cap \prod_{n\in\mathbb N}U_n$, we have $d_n(e_n,f_n)<\frac{\epsilon}{4}$ for all $n\in\{1,2,\dots,m\}$. Fix $n>m$. Then
    \begin{align*}
        d_n(e_n,f_n) &\leq  d_n(e_n,e_m)+d_n(e_m,f_m)+d_n(f_m,f_n)\\
                    &\leq  \sum_{j=m+1}^{n}\frac{a}{2^{j}}+d_m(e_m,f_m)+\sum_{j=m+1}^{n}\frac{a}{2^{j}}\\
                    & < \frac{\epsilon}{4}+\frac{\epsilon}{4}+\frac{\epsilon}{4}\leq \frac{3\epsilon}{4}
    \end{align*}

    Since $\sup\{d_n(e_n,f_n)\mid n\in\mathbb N\}<\epsilon$, we have $(f_n)\in B_{d}((e_n),\epsilon)$. Thus the desired inclusion is proved and $d$ is a metric compatible with the inverse limit topology on $E$.

     The monotone retraction of $E_i$ onto $E_{i-1}$ simply collapses the closure of each component $C$ of $E_i\backslash E_{i-1}$ to the point $C\cap E_{i-1}$, i.e., $\rho_i$ is the closest point projection map from $E_i$ to $E_{i-1}$.  Thus an arc in $E_i$ is mapped by $\rho_i$ to an arc (possibly degenerate) in $E_{i-1}$ by collapsing some initial and some terminal part of the arc.  Thus $\rho_i$ is $1$-Lipschitz. Notice that if $(e_i), (f_i)\in E_j$ for some $j$, then $e_i = e_j$ and $f_i= f_j$ for all $i\geq j$ which implies $d\big((e_i), (f_i)\big) = d_j(e_j, f_j)$.  Hence, $d|_{E_i\times E_i}= d_i$.

     Fix $n\in\mathbb{N}$.  Then, for $(e_i)\in E$,

     \[
        d\big(\varrho_n\big((e_i)\big),(e_i)\big) =  \sup_{i>n} d_i(e_n, e_i) \leq \sum_{j=n+1}^{i}d_j(e_j,e_{j-1}) \leq  \sum_{j=n+1}^{i}\frac{a}{2^n}\leq \frac{a}{2^n}.
    \]

Thus $\varrho_n$ converges uniformly to the identity map.
\end{proof}

\begin{lemma}\label{lem: limit is r-tree homotopic}
Let $E_i$ be a nested sequence of dendrites with geodesic metrics $d_j$ such that $d_{i}|_{E_{i-1}\times E_{i-1}} = d_{i-1}$ and suppose $a>0$. Suppose that we have maps $\pi_i:[0,1]\to E_i$ and $g_i: E_i\to \mathbb R^2$ for $i\geq 1$ and maps $\rho_i:E_i\to E_{i-1}$ for $i\geq 2$ with the following properties.

    \begin{enumerate}
        \item $\rho_i:E_i\to E_{i-1}$ is a monotone retract of dendrites
        \item $\sup\limits_{e\in E_i}d_i\big(\rho_i(e), e\big)\leq \frac{a}{2^i}$
        \item\label{fixed ends}$\pi_i(r) = \pi_j(r)$ for all $i,j$ and $r\in \{0,1\}$
        \item\label{cauchy} $\sup_t d_i\big(\pi_{i-1}(t), \pi_i(t)\big)\leq \frac{a}{2^{i-1}}$.
        \item $g_i: E_i\to \mathbb R^2$ is $L$-Lipschitz.
        \item $g_i|_{E_{i-1}} = g_{i-1}$
    \end{enumerate}

    For $E = \varprojlim(E_i, \rho_i)$,  the maps $\pi_i$ converge uniformly to a map $\pi: [0,1]\to E$ such that $\pi(r) = \pi_1(r)$ for $r\in \{0,1\}$.  The maps $g_i\circ\varrho_i$ converge uniformly to an $L$-Lipschitz map $g: E \to \mathbb R^2$ such that $g|_{E_i} = g_i$.  In particular, $g\circ\pi$ is $\mathbb R$-tree homotopic to $g_1\circ \pi_1$
\end{lemma}

\begin{proof}
    By Lemma \ref{lem: the dendrite}, $E = \varprojlim(E_i, \rho_i)$ admits a metric $d$ such that  $d|_{E_i\times E_i}= d_i$ (where we identify $E_i$ with its natural embedding in $E$) and the projections $\varrho_i: E\to E_i$ converge uniformly to the identity map on $E$. By (\ref{cauchy}), the maps $\pi_i$ form a Cauchy sequence (in the sup metric) of continuous maps $[0,1]\to E$.  Thus they converge uniformly to a map $\pi:[0,1]\to E$ and (\ref{fixed ends}) implies that $\pi(r)= \pi_1(r)$ for $r\in \{0,1\}$.

    Let $e= (e_i)\in E$.  Then for $k>j$, we have

    \begin{align*}
        d\big(g_k\circ\varrho_k(e), g_j\circ\varrho_j(e)\big) &=d\big(g_k(e_k), g_j(e_j)\big)\\
                    &=d\big(g_k(e_k), g_k(e_j)\big)\leq  \sum\limits_{i=j+1}^{k}d\big(g_k(e_i), g_k(e_{i-1})\big)\\
                    &\leq\sum\limits_{i=j+1}^{k}L\cdot d_k\big(e_i, e_{i-1})\leq\sum\limits_{i=j+1}^{k}L\frac{a}{2^i} \leq L\frac{a}{2^j}.
    \end{align*}

    Thus $g_k\circ\varrho_k$ is a Cauchy sequence of maps $E\to \mathbb{R}^2$ (in the sup metric) and therefore converge to a map $g: E\to \mathbb{R}^2$.  Notice that $g_k\circ\varrho_k|_{E_k} = g_k$ for all $k$.  Hence for all $\ell\geq k$, $g_\ell\circ\varrho_\ell|_{E_k} = g_k$.  Thus $g|_{E_k} = g_k$ which implies that $g\circ \pi_1 = g_1\circ \pi_1$.  Since $E$ is a dendrite, $\pi_1(0)= \pi(0)$, and $\pi_1(1) = \pi(1)$, the path $g\circ \pi_1$ is $\mathbb R$-tree homotopic to $g\circ \pi$. Hence $g_1\circ \pi_1$ is $\mathbb R$-tree homotopic to $g\circ \pi$.

    Since $\varrho_i$ is $1$-Lipschitz and $g_i$ is $L$-Lipschitz, $g_i\circ \varrho_i$ is $L$-Lipschitz which implies that $g$ is $L$-Lipschitz.

\end{proof}

\section{Non-transitivity of tree-like equivalent}\label{sec: the bow on top}

\begin{theorem}\label{finalthm}
    There exists paths $\alpha$, $\beta$, and $\gamma$ in $\mathbb{R}^2$ such that $\alpha$ is tree-like equivalent to $\gamma$, $\beta$ is tree-like equivalent to $\gamma$, but $\alpha$ is not tree-like equivalent to $\beta$.  In addition, $\alpha$ and $\beta$ can be chosen to be tree-like reduced paths.
\end{theorem}

\begin{proof}
Let $\pi_n: [0,1]\to E_n$, $\tilde\pi_n: [0,1]\to \tilde E_n$,  $g_n: E_n\to \mathbb{R}^2$, and $\tilde g_n: \tilde E_n\to \mathbb{R}^2$ be the maps and simplicial trees constructed in Section \ref{sec: construction}.

Notice that $g_1\circ\pi_1$ parametrizes the two legs of the triangle $T=\triangle abc$ and $\tilde g_1\circ\tilde \pi_1$ parameterizes the hypotenuse. Condition (\ref{collapse}) of the inductive hypothesis guarantees that $\sup\limits_{e\in E_i}d_i\big(\rho_i(e), e\big)\leq \frac{1}{2^{i-1}}$ and $\sup\limits_{e\in\tilde  E_i}d_i\big(\tilde \rho_i(e), e\big)\leq \frac{\sqrt 2}{2^{i-1}}$. Lemma \ref{lem: the dendrite} gives us two dendrites $E = \varprojlim(E_i, \rho_i)$ and $\tilde E = \varprojlim(\tilde E_i, \tilde \rho_i)$ with metrics $d$ and $\tilde d$ such that $d|_{E_i\times E_i}= d_i$, $d|_{\tilde  E_i\times\tilde  E_i}=\tilde  d_i$, such that the projections $\varrho_i: E\to E_i$, $\tilde \varrho_i: \tilde E\to \tilde E_i$ converge uniformly to the respective identity map. Then Lemma \ref{lem: limit is r-tree homotopic} gives us maps $\pi:[0,1]\to E$, $\tilde \pi:[0,1]\to \tilde E$, $g: E\to \mathbb{R}^2$, and  $\tilde g: \tilde E\to \mathbb{R}^2$ such that $g\circ\pi$ is $\mathbb R$-tree homotopic to $g_1\circ \pi_1$ and $\tilde g\circ\tilde \pi$ is $\mathbb R$-tree homotopic to $\tilde g_1\circ \tilde \pi_1$.

Let $d^{\s}$ denote the uniform metric on paths in $\mathbb{R}^2$. Since $(\pi_n, \tilde \pi_n, g_n, \tilde g_n) $ is a parametrization of $A_n$ relative to $S=\{0, \frac{1}{4^{n-1}}, \frac{2}{4^{n-1}}, \dots, 1\}$ and the triangles of $A_n$ have diameter $\frac{\sqrt{2}}{2^{n-1}}$, we have that $d^{\s}\big(g_n\circ\pi_n,\tilde g_n\circ\tilde \pi_n \big)\leq \frac{\sqrt2}{2^{n-1}}$.

Then $g_k\circ \pi_k = g\circ \pi_k$, since $g|_{E_k} = g_k$. So $d^{\s}(g_k\circ \pi_k, g\circ \pi) = d^{\s}(g\circ \pi_k, g\circ \pi)$.  Since $\pi_k$ converges uniformly to $\pi$ and $g$ is $1$-Lipschitz, $g_k\circ \pi_k$ converges uniformly to $g\circ \pi$.  Similarly, $\tilde g_k\circ \tilde \pi_k$ converges uniformly to $\tilde g\circ\tilde  \pi$. Since $d^{\s}\big(g_n\circ\pi_n,\tilde g_n\circ\tilde \pi_n \big)\leq \frac{\sqrt2}{2^{n-1}}$, we have that $g\circ \pi= \tilde g\circ\tilde  \pi$.

Let $\alpha = g_1\circ\pi_1= g\circ\pi_1$, $\beta = \tilde g_1\circ\tilde \pi_1= \tilde g\circ\tilde \pi_1$, and $\gamma = g\circ \pi=\tilde g\circ\tilde  \pi$. Then we can consider $\pi_1*\overline \pi:[0,1]\to E$, $\tilde \pi_1*\overline {\tilde \pi}:[0,1]\to \tilde E$, $g: E\to \mathbb{R}^2$, and  $\tilde g: \tilde E\to \mathbb{R}^2$.  Since $g$ and $\tilde g$ are 1-Lipschitz, we have that $\alpha*\overline\gamma = g\circ (\pi_1*\overline \pi)$ and $\beta*\overline\gamma = \tilde g\circ (\tilde \pi_1*\overline{\tilde  \pi})$ are both tree-like by Lemma  \ref{lem: height functions from dendrites}. Since $\alpha*\overline\beta$ is a simply closed curve it is not $\mathbb R$-tree nullhomotopic and, hence, by Lemma \ref{lem: dendrites from height functions}, it is not tree-like.  Note that $\alpha$ and $\beta$ parameterize arcs, so they are reduced paths.  This completes the proof.




\end{proof}

\begin{corollary}[Theorem 1.4 of \cite{BrazasConnerFabelKent_preprint}]
If $X$ is an arbitrary topological space and $\phi: [0,1]\to X$ and $\theta:[0,1]\to X$ are path-homotopic, then there exists a path $\psi:[0,1]\to X$ such that $\phi$ is $\mathbb R$-tree homotopic to $\psi$ and $\theta$ is $\mathbb R$-tree homotopic to $\psi$.
\end{corollary}

\begin{proof}
The construction in the proof of Theorem \ref{finalthm} gives path $\alpha:[0,1]\to T$ parameterizing the legs of $T$, path $\beta:[0,1]\to T$ parameterizing the hypotenuse of $T$, and path $\gamma:[0,1]\to \conv(T)$ in the convex hull of $T$ such that $\alpha\simeq_{\mathbb{R}}\gamma$ and $\beta\simeq_{\mathbb{R}}\gamma$ in $\conv(T)$. Since $\phi$ is path-homotopic to $\theta$, there exists a map $h:\conv(T)\to X$ such that $h\circ \alpha=\phi$ and $h\circ \beta=\theta$. Setting $\psi=h\circ\gamma$, we have $\phi\simeq_{\mathbb{R}}\psi$ and $\theta\simeq_{\mathbb{R}}\psi$.
\end{proof}

	\bibliographystyle{plain}
	\bibliography{bib}

\end{document}